\documentclass[a4paper,12pt]{amsart}
\usepackage{amsmath, amsthm, amstext, eucal}

\usepackage{amssymb, array, amsfonts}

\numberwithin{equation}{section}

\def \al{\alpha}

\def \ga{\gamma}

\def \ze{\zeta}

\def \L{\Lambda}
\def \O{\Omega}

\def \C{\mathbb{C}}

\def \Q{\mathbb{Q}}
\def \R{\mathbb{R}}
\def \Z{\mathbb{Z}}

\def\n{\nabla}

\def\1{1\!\!\!\!1}

\def\im{\operatorname{Im}}

\newcommand{\<}{\langle}
\renewcommand{\>}{\rangle}

\theoremstyle{plain}
\newtheorem{theorem}{\bf Theorem}[section]
\newtheorem{lemma}[theorem]{\bf Lemma}

\theoremstyle{remark}
\newtheorem{rem}[theorem]{\bf Remark}

\renewcommand{\le}{\leqslant}
\renewcommand{\ge}{\geqslant}

\title[On the "even" periodic Schr\" odinger operator]
{On the spectrum of an "even" periodic Schr\" odinger operator
with a rational magnetic flux}

\author{N.~D.~Filonov, A.~V.~Sobolev}
\address{St. Petersburg Department of Steklov Mathematical Institute\\
27 Fontanka\\ St. Petersburg 191023\\ Russia\\ and 
Department of Physics, St. Petersburg State University,
Ulyanovskaya 1, Petrodvorets, St. Petersburg 198504\\ Russia}
\address{Department of Mathematics\\ University College London\\
Gower Street\\ London\\ WC1E 6BT UK} 
 \email{filonov@pdmi.ras.ru}
\email{a.sobolev@ucl.ac.uk}

\keywords{ 
The magnetic Schr\"odinger operator, absolute continuity, rational flux}
\subjclass[2010]{Primary 35Q40; Secondary 35J10}

\date{}
\begin{document}

\begin{abstract}
We study the 
Schr\"odinger operator on $L_2(\mathbb R^3)$ 
with periodic variable metric, and periodic electric and magnetic fields. 
It is assumed that the operator is reflection symmetric and 
the (appropriately defined) 
flux of the magnetic field is rational. 
Under these assumptions it is shown that 
the spectrum of the operator is absolutely 
continuous. Previously known results on absolute continuity for periodic 
operators 
were obtained for the zero magnetic flux.  
\end{abstract}

\maketitle

\section{Introduction and results}

In the last two decades a good deal of attention was focused on the 
absolute continuity of self-adjoint 
periodic differential elliptic operators of second order 
%see \cite{?} 
in dimension $d\ge 2$, i.e. of the operators of the form 
\begin{equation}\label{ham:eq}
H = \sum_{j, l =1}^d(D_j - A_j)g_{jl}(D_l - A_l) + V, 
\ \
D_j = -i\partial_j, 
\end{equation}
with a periodic symmetric positive-definite matrix 
$\{g_{jl}\} = \mathbf G$, and coefficients  
$\mathbf A = \{A_l\}$, $V$ which we interpret as 
magnetic and electric potentials respectively.   
If all the coefficients in \eqref{ham:eq} are periodic and satisfy 
suitable integrability and/or smoothness conditions, then 
the operator $H$ is known to be absolutely continuous 
for $d=2$. If $\mathbf G(\mathbf x) = g(\mathbf x)\mathbf I$ with a positive 
function $g$ then this conclusion extends to 
arbitrary $d \ge 2$. 
We do not provide a thorough bibliographical account 
and refer e.g. to \cite{BSu2}, \cite{KL} and  \cite{SS} for more detailed references. 

The case of general variable $\mathbf G$
in dimensions $d\ge 3$ remains unassailable, but there are some partial results.  
First, if the matrix $\mathbf G$ is not
smooth then the spectrum of $H$ may not be absolutely
continuous, see \cite{F}. Second, 
in L. Friedlander's paper \cite{Fr} 
the absolute continuity was obtained for smooth variable matrix $\mathbf G$ 
and smooth $\mathbf A$, $V$ for all dimensions $d\ge 2$ under the condition 
that the operator $H$ is reflection symmetric. 
Later the smoothness assumptions 
were relaxed by N. Filonov, M. Tikhomirov in \cite{FT}.  

In this note we address another open question of the theory: 
when is the spectrum absolutely continuous if instead of the magnetic potential 
$\mathbf A$ we assume that the magnetic field $\mathbf B = 
\operatorname{curl} \mathbf A$ is periodic? 
The traditional methods used to study the spectra of periodic operators 
are not directly applicable. However, under the additional condition 
of the reflection symmetry one can still use the ideas of \cite{Fr} and \cite{FT}.   
We concentrate on the physically relevant case $d=3$. 
Note that 
the case $d=2$ is also of interest but 
the requirement of the reflection symmetry automatically implies that 
the constant component of the magnetic field is zero, i.e. the 
magnetic potential itself becomes periodic. Thus for $d = 2$ 
the Friedlander's method 
would give no new information.   
At this point we should note that 
in general (i.e. without reflection symmetry), 
the two-dimensional case is dramatically different from 
the three-dimensional one. 
It suffices to observe that in the absence of electric field 
for $d = 3$ a constant magnetic field induces 
absolutely continuous spectrum, whereas 
for $d=2$ the spectrum consists of 
equidistant eigenvalues of infinite multiplicity, called 
Landau levels, see \cite{LL}. Thus for $d=2$ 
mechanisms responsible for the possible 
formation of the absolute continuous 
spectrum (e.g. with non-trivial periodic $V$) 
are very different. 

Let us proceed to the precise formulations. 
The operator $H$ is defined via the quadratic form 
\begin{equation}
\label{00}
h[u] = \int_{\R^3} 
\langle \mathbf G(\mathbf x) \left(-i\n 
- \mathbf A(\mathbf x)\right) u(\mathbf x), 
\left(-i\n 
- \mathbf A(\mathbf x)\right) u(\mathbf x)
\rangle d\mathbf x
+  \int V(\mathbf x) |u(\mathbf x)|^2 d\mathbf x, 
\end{equation}
with the domain $D[h] = C_0^\infty (\R^3)$
in the Hilbert space $L_2(\R^3)$. 
The coefficient $\mathbf G = \{g_{jl}(\mathbf x)\}, j,l = 1, 2, 3$,
is a symmetric matrix-valued function 
with real-valued entries
$g_{jl}(\mathbf x)$ which satisfies the conditions
\begin{align}
c|\boldsymbol\xi|^2\le &\ \langle\mathbf G(\mathbf x)\boldsymbol\xi,
\boldsymbol\xi\rangle\le C|\boldsymbol\xi|^2,  
\quad \forall \boldsymbol\xi \in
\R^3,\ \textup{a.a.} \ \   \mathbf x\in \R^3, \label{g:eq}
\\[0.2cm]
\mathbf G\in &\ \textup{Lip}(\R^3). \label{lip:eq}
\end{align} 
Here and everywhere below
by $C$ and $c$ with or
without indices we denote various positive 
constants whose precise value is unimportant.
The vector-field $\mathbf A$ and the function $V$ satisfy the conditions 
\begin{equation}
\label{01}
\mathbf A \in L_{p,\textup{\tiny loc}} (\R^3, \R^3), \qquad 
V \in L_{3/2,\textup{\tiny loc}} (\R^3),
\end{equation}
with $p=3$. 
Under the assumptions \eqref{g:eq}, \eqref{01} with $p=3$, and 
that $V$ is periodic, the form \eqref{00}  
is semibounded from below and closable (see e.g. \cite[\S 2]{MR}).
We denote by $H$ the self-adjoint operator in $L_2(\R^3)$
which corresponds to the closure of the form $h$. We write it formally as 
\begin{equation}
\label{03}
H = \langle( -i \n - \mathbf A), \mathbf G( -i \n - \mathbf A)\rangle + V.
\end{equation}
Since we assume that the magnetic field 
$\mathbf B(\mathbf x) = \operatorname{curl} \mathbf A(\mathbf x)$ is periodic, 
the magnetic potential can be represented in the form 
\begin{equation*}
\mathbf A(\mathbf x) = \mathbf a_0(\mathbf x) 
%\frac{1}{2} \ \mathbf B_0 \times \mathbf x 
+ \mathbf a(\mathbf x),
\end{equation*}
where $\mathbf a_0$ is a linear magnetic potential 
associated with the constant component 
$\mathbf B_0 = \operatorname{curl} \mathbf a_0(x)$ 
of the magnetic fields, and 
$\mathbf a$ is a periodic vector-potential. We align $\mathbf B_0$ 
with the positive direction of the $x_3$-axis, and choose 
for $\mathbf a_0(\mathbf x)$ the gauge 
$(-b x_2, 0, 0), b = |\mathbf B_0|\ge 0$, so that 
$\mathbf B_0  = (0, 0, b)$ and 
\begin{equation}\label{11}
\mathbf A(\mathbf x) =  (-b x_2, 0, 0) 
+ \mathbf a(\mathbf x). 
\end{equation}
We assume that with this choice of 
coordinates the matrix-valued function $\mathbf G$, the potentials $V$ and 
$\mathbf a$ are $(2\pi\mathbb Z)^3$-periodic:
\begin{equation}
\label{02}
\mathbf G(\mathbf x+2\pi \mathbf n) 
= \mathbf G(\mathbf x),\ 
V(\mathbf x+2\pi \mathbf n) 
= V(\mathbf x),\ 
\mathbf a(\mathbf x+2\pi \mathbf n) 
= \mathbf a(\mathbf x), 
 \quad \forall \mathbf n \in \Z^3.
\end{equation}
Furthermore, to ensure that the operator \eqref{03} is partially diagonalizable  
via the Floquet-Bloch-Gelfand decomposition, we assume that the flux of the 
constant component $  \mathbf B_0$ is integer, i.e. 
 \begin{equation}\label{13}
 \frac{1}{2\pi}\underset{(-\pi, \pi)^2} \int
 |\mathbf B_0| dx_1 dx_2 = 2\pi b \in \mathbb Z_+ = \{0, 1, \dots\}.
 \end{equation}
To describe the symmetry of the operator $H$ introduce the reflection 
map $\mathbf R: \R^3\to\R^3$: 
$$
\mathbf R(x_1, x_2, x_3) = (x_1, x_2, -x_3),
$$
and the associated operation on $u\in L_2(\R^3)$:
\begin{equation}\label{J:eq}
(J u)(\mathbf x) = u(\mathbf R\mathbf x).
\end{equation}
It is straightforward to check that $H$ commutes with $J$ 
if $\mathbf G$, $\mathbf a$ and $V$ satisfy the conditions 
\begin{equation}\label{even:eq}
\mathbf G(\mathbf R\mathbf x) 
= \mathbf R\mathbf G(\mathbf x)\mathbf R,\ 
\mathbf A(\mathbf R\mathbf x) = \mathbf R\mathbf A(\mathbf x), 
\quad V(\mathbf R\mathbf x) = V(\mathbf x), \quad \ \  
\textup{a.a.}\  \ \mathbf x \in \R^3.
\end{equation}
Obviously the symmetry condition for $\mathbf A$ is equivalent 
to that for $\mathbf a$.

The next theorem constitutes the main result of the paper. 

\begin{theorem}\label{t1}
Assume that the matrix $\mathbf G$, the potentials 
$\mathbf A, V$ satisfy the conditions \eqref{g:eq}, \eqref{lip:eq} and  
\eqref{01} with $p > 3$. Assume also that  
\eqref{11}, \eqref{02}, \eqref{13} and \eqref{even:eq} are satisfied. 
Then the spectrum of the operator \eqref{03} is absolutely continuous.  
\end{theorem}

Throughout the paper we always assume the periodicity 
\eqref{02}. As a special case this allows a constant magnetic field 
i.e. $\mathbf a = 0$. 
With regard to the regularity, 
we normally need only \eqref{g:eq} and \eqref{01} with $p = 3$. 
The assumptions \eqref{lip:eq} and $p >3$ are required  
only once when employing the 
unique continuation argument, see Lemma \ref{l54}. Recall that if 
\eqref{lip:eq} is not satisfied, the spectrum may not be 
absolutely continuous, see \cite{F}.

Note that the condition \eqref{13} can be replaced by $2\pi b\in \Q$. 
This case reduces to that of an 
integer flux by taking an appropriate sublattice of $\Z^3$ and rescaling.  
If the flux is irrational we cannot say anything about the nature of the spectrum.

As mentioned earlier, one can state a theorem similar to 
Theorem \ref{t1} 
in the two-dimensional case as well. 
However, in this case the reflection symmetry would 
imply that $b = 0$, see \eqref{11}, and hence such a theorem 
would not say anything new compared to the known results. 

Theorem \ref{t1} can be conceivably generalized to 
 arbitrary dimensions $d \ge 3$ with the standard 
changes to the conditions \eqref{01}. 
We have chosen not to clutter the presentation with these details but 
to focus on the 
lowest dimension where the reflection symmetry leads 
to  a non-trivial effect. 

As mentioned earlier, L. Friedlander (see \cite{Fr}) was the first to notice 
how the reflection symmetry can be used to establish the absolute continuity of $H$.  
We follow the paper \cite{FT} where Friedlander's scheme was implemented 
with relaxed regularity assumptions. On the other hand our proof is simpler and somewhat shorter 
than that of \cite{FT}, and we consider it worthy of dissemination. 

{\bf Acknowledgments.} This work was supported 
by the grant RFBR 11-01-00458 (N.D.F.) and by the EPSRC grant EP/J016829/1 (A.V.S.).
%%%%%%%%%%%%%%%%%%%%%%%%%%%%%%%%%%%%%%%%%%%%%%%%%%%%%%%%%%%%%%%%%%%%%%%%%%

\section{Floquet-Bloch-Gelfand transformation}
\label{s2}

Denote by $\O$ the interior of the standard 
fundamental domain of the lattice $\Gamma = \mathbb Z^3$: 
$\O = (-\pi,\pi)^3$. We also need separate 
notation for the top and bottom faces of this cube:
\begin{equation*}
\Lambda_\pm = \{\mathbf x\in \R^3: 
\hat{\mathbf x}\in (-\pi, \pi)^2, \ x_3 = \pm \pi\},\ \hat{\mathbf x} = (x_1, x_2).
\end{equation*}
The interior of the fundamental domain 
of the dual lattice is denoted $\O^\dagger = (0, 1)^3$. 

The Floquet-Bloch-Gelfand transform is defined as the operator 
$$
(Uf) (\mathbf x, \mathbf k) = \sum_{\mathbf n\in\Z^3} 
e^{-i\mathbf k\cdot (\mathbf x+2\pi \mathbf n)}
e^{i2\pi b n_2 x_1} f(\mathbf x+2\pi \mathbf n),\ 
\mathbf x\in \O, \mathbf k\in \O^\dagger,
$$
for functions $f \in C_0^\infty (\R^3)$. 
It is clear that $Uf\in C^\infty(\overline\O\times\overline{\O^\dagger})$.  
Moreover, the function $v(\ \cdot\ ) = Uf(\ \cdot\ ; \mathbf k)$ is periodic in $x_1$  
(due to the condition \eqref{13}), and in $x_3$: 
\begin{align}
v(-\pi, x_2, x_3) = v(\pi, x_2, x_3),\label{per1:eq} \\
v(x_1, x_2, - \pi) = v(x_1, x_2, \pi).\label{per2:eq} 
\end{align} 
It is quasiperiodic in $x_2$:
\begin{equation}\label{quasi:eq}
v(x_1, \pi, x_3) = e^{-i2\pi bx_1} v(x_1, -\pi, x_3). 
\end{equation}
A direct calculation shows 
that the transform $U$ 
can be extended to $L_2(\R^3)$ as a unitary operator
\[
U : L_2 (\R^3) \to \int_{\O^\dagger}^\oplus L_2(\O) d\mathbf k .
\]
For each $\mathbf z\in\mathbb C^3$ introduce the quadratic form 
\begin{equation}\label{hz:eq}
h(\mathbf z) [v] = \int_\O 
\left(\left\< \mathbf G(\mathbf x)(-i\n + \mathbf z - \mathbf A(x)) v(\mathbf x), 
(-i\n + \overline{\mathbf z} - \mathbf A(\mathbf x)) v(\mathbf x)\right\>
+ V(\mathbf x) |v(\mathbf x)|^2 \right) d\mathbf x. 
\end{equation}
Under the conditions \eqref{g:eq}, \eqref{01} with $p=3$ 
the potentials $\mathbf A$ and $V$ 
induce on $C^\infty(\overline\O)$ a perturbation which is infinitesimally bounded 
by the standard Dirichlet form, and hence 
\begin{equation}\label{eqform:eq}
C_0^{-1}\|v\|_{H^1(\O)}^2\le  |h(\mathbf z)[v] | + C\|v\|_{L_2(\O)}^2
\le  C_0 \|v\|_{H^1(\O)}^2,
\end{equation}
with some positive constants $C = C(\mathbf z)$ 
and $C_0 = C_0(\mathbf z) >1$ uniformly 
in $\mathbf z$ on a compact subset of $\mathbb C^3$. Thus 
\eqref{hz:eq} naturally extends to 
all $v\in H^1(\O)$ as a closed form. 
In order to relate this form to the form \eqref{00} we consider 
\eqref{hz:eq} on a smaller domain. 
It is convenient to introduce a special notation for the function spaces 
with the conditions \eqref{per1:eq} and \eqref{quasi:eq}:
\begin{equation}\label{tilde:eq}
W^1 = \{u\in H^1(\O): \ \textup{$u$ satisfies \eqref{per1:eq} and \eqref{quasi:eq}}\}.
\end{equation}
Now we consider the form \eqref{hz:eq} on the domain 
\begin{equation*}
 D[h(\mathbf z)] = D[h(\mathbf 0)] = 
 \{v\in W^1: v \ \textup{satisfies \eqref{per2:eq}}\}.
\end{equation*}
Clearly the form \eqref{hz:eq} is closed on 
$D[h(\mathbf 0)]$ and analytic (quadratic) in $\mathbf z\in\mathbb C^3$. 
One checks directly that  
\begin{equation}\label{dirform:eq}
h[v] = \int_{\O^\dagger} h(\mathbf k)[(Uv)(\ \cdot\ , \mathbf k)] d\mathbf k.
\end{equation}
for any $v\in D[h(\mathbf 0)]$. 
The form $h(\mathbf z)[v]$ is sectorial, 
i.e. for a suitable number $\gamma = \gamma(\mathbf z)\in\R$, 
\begin{equation*}
\textup{Re} \ h(\mathbf z)[v]\ge -\gamma\|v\|^2,\   
|\textup{Im}\  h(\mathbf z)[v]| 
\le C\bigl(\textup{Re} \ h(\mathbf z)[v] + \gamma\|v\|^2\bigr),\ 
v\in D[h(\mathbf 0)],
\end{equation*}
with some positive constant 
$C = C(\mathbf z)$ uniformly in $\mathbf z$ on a compact subset of $\mathbb C^3$. 
Hence it defines a sectorial operator (m-sectorial in T. Kato's terminology, see 
\cite{K}) 
which we denote by $H(\mathbf z)$. 
As the form $h(\mathbf z)$ 
is compact in $H^1(\O)$ the resolvent of $H(\mathbf z)$ is compact 
whenever it exists.  
For the values $\mathbf k\in\R^3$ the operator $H(\mathbf k)$ is self-adjoint: 
$H(\mathbf k)=H(\mathbf k)^*$. 
In view of \eqref{dirform:eq} the following unitary equivalence
\begin{equation}
\label{21}
U H U^* = \int_{\O^\dagger}^\oplus H(\mathbf k)\, d\mathbf k 
\end{equation}
holds. Although this formula requires only the values $\mathbf k\in\O^\dagger$ 
it is important for us to have the operator 
$H(\mathbf z)$ defined for $\mathbf z\in\mathbb C^3$. Sometimes 
we use the notation $\mathbf z = (\hat{\mathbf z}, z_3)$, with 
$\hat{\mathbf z} = (z_1, z_2)$. 

In order to prove Theorem \ref{t1} it is sufficient to show 
that $H$ has no eigenvalues, see \cite{FilSob}. 
The proof of this fact reduces to the analysis of the 
following boundary value problem for a function $u\in W^1$ with the form 
$h_0 = h(\mathbf 0)$ and a number $\ze\in\mathbb C$: 
\begin{equation}\label{per3:eq}
u\in W^1, \quad \left.u \right|_{\Lambda_+} = \zeta \left.u\right|_{\Lambda_-},
\end{equation}
\begin{equation}\label{24}
h_0[u, w] = 0,\ \quad \forall w\in W^1, \ \ \textup{s.t.}\ \ 
\overline{\zeta} \left.w\right|_{\Lambda_+} =  \left.w\right|_{\Lambda_-}.
\end{equation}
Theorem \ref{t1} is derived from the next theorem:

\begin{theorem}\label{t2} 
Suppose that $\mathbf G$, $\mathbf A$ and $V$ satisfy the conditions of Theorem \ref{t1}.  
Let $X\subset\mathbb C$ be the subset of the complex plane consisting of 
the points $\ze$ such that
\begin{enumerate}
\item
$\im \ze \neq 0$, $|\ze| \neq 1$,
\item
there exists at least one function $u\in W^1, u\not\equiv 0$  
satisfying \eqref{per3:eq} and \eqref{24}.
\end{enumerate}
Then $X$ is at most finite. 
\end{theorem}

\begin{proof}[Derivation of Theorem \ref{t1} from Theorem \ref{t2}] 
 Use Theorem \ref{t2} with the potentials 
$\mathbf A - (\hat{\mathbf k}, 0)$ and $V - \lambda$, where
$\hat{\mathbf k}\in (0,1)^2$ and  $\lambda\in \R$. Clearly 
these new potentials satisfy the conditions of Theorem \ref{t1} as well. 
Let $u$ be a non-trivial solution of the problem \eqref{per3:eq} 
\eqref{24} with 
some $\zeta\in X$, so that 
\begin{equation}\label{k:eq}
h\bigl((\hat{\mathbf k}, 0)\bigr)[u, w] = \lambda(u, w). 
\end{equation}
Let $k\in \mathbb C$ be such that $\ze = e^{i2\pi k}$. 
Re-denoting 
\begin{equation*}
v(\mathbf x) = e^{-ik x_3} u (\mathbf x),
\quad \eta(\mathbf x) = e^{-i\overline k x_3} w(\mathbf x) 
\end{equation*}
we reduce \eqref{k:eq} and the boundary conditions \eqref{per3:eq}  
to the following equation for the function $v\in D[h_0]$: 
\begin{equation*}
h\bigl((\hat{\mathbf k}, k)\bigr) 
[v, \eta] = \lambda(v, \eta), \quad \forall \eta \in D[h_0].
\end{equation*}
This means that $\lambda$ is an eigenvalue of $H(\hat{\mathbf k}, k)$ for all 
$k$ such that $\exp(i2\pi k)\in X$.
Since $H(\mathbf z)$ has compact resolvent, by the  
 analytic Fredholm alternative (see \cite{K}, Theorems VII.1.10, VII.1.9), 
 the finiteness of the set $X$ implies that 
the measure of the set 
\begin{equation*}
\{k\in (0, 1): \lambda\in \sigma(H(\hat{\mathbf k}, k)) \}
\end{equation*}
equals zero for any $\hat{\mathbf k}\in (0, 1)^2$.  
Consequently the measure of the set 
\begin{equation*}
\{\mathbf k\in \O^\dagger: \lambda\in \sigma(H({\mathbf k})) \}
\end{equation*}
is also zero. Thus the point $\lambda$ is not an eigenvalue of the operator \eqref{21}. 
Since $\lambda\in\R$ is arbitrary this implies that 
the operator $H$ has no eigenvalues. According to \cite{FilSob} this ensures that 
the spectrum of $H$ is absolutely continuous, as required. 
\end{proof}

%%%%%%%%%%%%%%%%%%%%%%%%%%%%%%%%%%%%%%%%%%%%%%%%%%%%%%%%%%%%%%%%%%%%%%%%%%%%%%%%
\section{Associated boundary-value problem}
\label{s3}

We begin the analysis of the system \eqref{per3:eq}, \eqref{24} 
with introducing the subspaces
\begin{align*}
W^{1,0} = &\ \{ v \in W^1 : 
\left. v \right|_{\Lambda_+} = \left. v \right|_{\Lambda_-} = 0 \},\\[0.2cm]
W^1_+ = &\ \{ u \in W^1 : \left. u \right|_{\Lambda_+} = 0\},
\end{align*}
with the standard $H^1$-inner product. Now define the subspaces 
\begin{equation*}
N = \left\{ v \in W^{1,0} : 
h_0 [v, w] = 0, \ \forall w \in W^{1,0} \right\}, 
\end{equation*}
\begin{equation*}
M = \{ u\in W^1_+ : h_0 [u, v] = 0,
\ \forall v \in N \},
\end{equation*}
and
\begin{equation*}
Z = \{ u \in W^1_+ : h_0 [u,w] = 0, \quad \forall w \in W^{1,0}, \ \
u \perp N \}
\end{equation*}
The subspace $Z$ consists of weak solutions $u\in W^1_+$ 
of the equation $H u = 0$ which are orthogonal 
to $N$. By definition 
of $M$ we automatically have $Z, W^{1, 0}\subset M$.

First of all consider the following 
boundary-value problem. 

\begin{lemma}\label{Z:lem} 
Let the conditions \eqref{g:eq} and  \eqref{01} 
with $p=3$ be satisfied. 
Then for any function $u\in M$ the system 
\begin{equation}
\label{31}
\begin{cases}
h_0 [\phi ,w] = 0, \quad \forall w \in W^{1,0}, \\
\phi - u\in W^{1, 0}. 
\end{cases}
\end{equation}
is solvable for the function $\phi\in W^1_+$. 
The solution is unique under the condition 
$\phi\perp N$.
Moreover, $\dim N <\infty$. 
\end{lemma}

\begin{proof}
The system is studied in the standard way. Namely, 
the function $\psi = \phi - u\in W^{1, 0}$ satisfies 
\begin{equation}\label{hom:eq}
h_0 [\psi, w] = - h_0[u, w], \quad \forall w \in W^{1,0}.
\end{equation}
Referring to \eqref{eqform:eq} 
introduce on $W^{1, 0}$ the inner product 
\begin{equation*}
(f, g)_1 = h_0[f, g] + \gamma (f, g) 
\end{equation*}
choosing $\gamma \ge 0$ in such a way that the induced norm $\|f\|_1$ is 
equivalent to the standard $H^1$-norm. The $L_2$-inner product is an example of 
a symmetric compact form in $H^1$, and hence there is a compact self-adjoint 
operator $T:W^{1, 0}\to W^{1, 0}$ 
such that   $(f, g) = (T f, g)_1, f, g\in W^{1, 0}$. 
As a result, the left-hand side of \eqref{hom:eq} 
rewrites as $((I-\ga T) \psi, w)_1$. 
The right-hand side of \eqref{hom:eq} is a continuous linear functional 
of $w\in W^{1, 0}$ so there is a function $q\in W^{1, 0}$ such that 
$- h_0[u, w] = (q, w)_1,\ \|q\|_1\le C \|u\|_1$. 
Thus \eqref{hom:eq} takes the form
\begin{equation}\label{fred:eq}
\psi - \gamma T \psi = q.
\end{equation}
Now it follows from the classical Fredholm Theory that \eqref{fred:eq} 
has a solution $\psi\in W^{1, 0}$ if and only if $(q, v)_1 = 0$
for all $v\in \ker (I-\ga T)$. 
Under this condition there is a unique solution $\psi_0$ 
satisfying the property $(\psi_0, v)_1 = 0$ 
for all $v\in \ker(I-\ga T)$, and this solution satisfies 
the bound $\|\psi_0\|_1\le C\| q\|_1$. 
Note that $N = \ker (I-\gamma T)$, so by 
definition of $q$ the equality $(q, v)_1 = 0$, 
$\forall v\in \ker (I-\gamma T)$ follows from the condition 
$u\in M$. Thus \eqref{hom:eq} is 
solvable and hence so is \eqref{31}. 
As $T$ is compact, it immediately follows that $\dim N < \infty$. 

Denote $\phi_0 = \psi_0 + u\in W^1_+$. Any other solution of \eqref{31} 
has the form  $\phi = \phi_0 + w$ with a suitable $w\in N$. 
If one demands that $\phi \perp N$ then $w = -\mathcal P \phi_0$ 
where $\mathcal P$ is the projection in $L_2(\O)$ on the 
finite-dimensional subspace $N$. Therefore such a solution $\phi\in W^1_+$ 
is uniquely defined, as required.  
\end{proof}
 
The following elementary lemma is 
crucial for us.

\begin{lemma}\label{l31} 
Let the conditions 
\eqref{g:eq}, \eqref{01} with $p =3$ be satisfied.  
Let the subspaces $M, Z$ be as defined above. 
Then the subspace $Z$ is non-trivial, and 
\begin{equation*}
M = Z \dot +\, W^{1,0}.
\end{equation*}
In other words, any function $u\in M$ is uniquely represented 
as the sum $\phi+w$ with some $\phi\in Z$ and $w\in W^{1, 0}$.
\end{lemma}

\begin{proof}   
By Lemma \ref{Z:lem} for any function $u\in M$ there is a  
solution $\phi$ of \eqref{31} orthogonal to $N$. 
Furthermore, $\phi$ is uniquely defined and $w = \phi - u\in W^{1, 0}$, 
so $M = Z \dot + \ W^{1, 0}$, as claimed. Recall that 
$\textup{codim} \ M < \infty $ in $W^1_+$ whereas 
$\textup{codim}\ W^{1, 0} = \infty$, so $M\not = W^{1, 0}$. This 
implies that $Z$ is non-trivial. 
\end{proof}

%%%%%%%%%%%%%%%%%%%%%%%%%%%%%%%%%%%%%%%%%%%%%%%%%%%%%%%%%%%%%%%%%%%

\section{The Dirichlet-Neumann forms}

\subsection{General facts} 

On the subspace $Z$ considered as a Hilbert space with the $H^1$-inner product 
introduce the forms 
\begin{equation*}
t_0[u, v] = h_0[u, v], \ \ t_1[u, v] = h_0[u, Jv],\ u, v\in Z, 
\end{equation*}
where $J$ is defined in \eqref{J:eq}. We call $t_0$ and $t_1$ 
the Dirichlet-Neumann forms. 
We list their properties in the following lemma.

\begin{lemma} 
Let the conditions 
\eqref{g:eq}, \eqref{01} with $p = 3$ be satisfied.  
Let $t_0, t_1$ be as defined above. Then 
\begin{enumerate}
\item
Both forms $t_0$ and $t_1$ are bounded on $Z$:
\begin{equation}\label{Mbound:eq}
|t_0 [\phi, \psi]| + |t_1 [\phi, \psi]|\le C\|\phi\|_{H^1}\   \|\psi\|_{H^1}. 
\end{equation}
\item
The form $t_0$ is Hermitian. If the condition \eqref{even:eq} is satisfied then 
$t_1$ is also Hermitian. 
\item
Let $\mathcal L\subset Z$ be a linear set such that  
$t_0[\phi]\le 0$ for all $\phi\in \mathcal L$. Then 
\begin{equation}\label{vari:eq}
\sup_{\mathcal L} \dim \mathcal L <\infty. 
\end{equation}
\end{enumerate}
\end{lemma} 

\begin{proof}
The bound \eqref{Mbound:eq} immediately 
follows from \eqref{eqform:eq}. 

The form $t_0$ is clearly Hermitian. If \eqref{even:eq} is satisfied, then 
\begin{equation*}
t_1 [u, v] = h_0 [u, J v] =
h_0 [J u, v] = \overline{t_1 [v, u]},
\quad \forall u, v \in Z,
\end{equation*}
i.e. $t_1$ is Hermitian. 

Consider the  
form $h_0$ on $H^1(\O)$, and recall that by \eqref{eqform:eq} 
with $\mathbf z = \mathbf 0$ it is closed and semibounded from below. 
Moreover, $H^1(\O)$ embeds into $L_2(\O)$ compactly, and hence 
the associated self-adjoint operator has discrete spectrum  accumulating at $+\infty$. 
The number of eigenvalues $n(\lambda)$ 
which are less than or equal to 
an arbitrary number $\lambda\in\R$ 
can be found in terms of the form $h_0$ in the standard way. 
Precisely, let $\mathcal L_{\lambda}\subset H^1(\O)$ be a linear set such that 
$h_0[u] \le \lambda\|u\|^2$ for all $u\in\mathcal L_\lambda$. 
Then 
\begin{equation*}
n(\lambda) = \max_{\mathcal L_\lambda}\dim \mathcal L_\lambda <\infty,
\end{equation*}
see \cite{BS}, Ch. 10, Theorem 2.3.
The form $t_0$ is the restriction of $h_0$ to the subspace $Z$, and hence 
\eqref{vari:eq} is a direct consequence of the above bound with $\lambda = 0$.   
\end{proof}

Instead of the solution space $Z$ we could have 
 considered the spaces of traces on the faces $\Lambda_-$, $\Lambda_+$. 
Then the forms $t_0$ and $t_1$ would correspond to 
two Dirichlet-Neumann operators $T_0$ and $T_1$ 
which map 
 the trace $\left.\phi\right|_{\Lambda_-},\ \phi\in Z$, into   
 the normal derivative of $\phi$ on the faces $\Lambda_-$ and $\Lambda_+$
respectively.  
This approach was adopted in the paper \cite{Fr}. 
We do not make explicit use of the Dirichlet-Neumann operators but it seems appropriate to use 
this terminology for the forms $t_0, t_1$.

\subsection{Reflection symmetry}

From now we assume that $\mathbf G$, $\mathbf A, V$ satisfy the symmetry condition 
\eqref{even:eq}. Thus using the operator $J$ defined in \eqref{J:eq} 
we get 
$$
h_0 [Ju, Jv] = h_0 [u, v],  
\quad \forall u, v \in H^1 (\O).
$$
 
Another consequence of the symmetry is that $J N = N$.

The next property is crucial for our argument. 

\begin{theorem}\label{t52}
Let the conditions 
\eqref{g:eq}, \eqref{lip:eq} and  \eqref{01} with $p > 3$ be satisfied.  
 Denote
\begin{equation}\label{kert1:eq}
\ker t_1 \equiv \left\{ u \in Z :
t_1 [u, v] = 0,  \ \forall v \in Z \right\}.
\end{equation}
If \eqref{even:eq} is satisfied then $\ker t_1 = \{0\}$. 
\end{theorem}

For the proof of this fact we need two lemmas.

\begin{lemma}
\label{l53}
Let ${\mathfrak H}$ be a Hilbert space, and 
$\ell$, $\ell_1, \dots, \ell_n$, $n < \infty$, 
be bounded linear functionals on ${\mathfrak H}$. 
If
\begin{equation}
\label{52}
\bigcap_{k=1}^n \ker \ell_k \subset \ker \ell,
\end{equation}
then the functional $\ell$ is a linear combination of the others:
$\ell = \sum_{k=1}^n \al_k \ell_k$ with some coefficients 
$\alpha_k$, $k= 1, 2, \dots, n$.
\end{lemma}

Although this fact is elementary we provide a proof for the sake of completeness. 

\begin{proof}
Let $z$, $z_1, \dots, z_n \in {\mathfrak H}$ be the 
uniquely defined vectors such that 
$$
\ell(x) = (x,z), \quad \ell_k(x) = (x, z_k), 
\ \ k = 1, \dots, n, \quad \forall x \in {\mathfrak H} .
$$
The condition \eqref{52} is equivalent to the following 
implication: 
if $ x \perp {\mathfrak L} = \operatorname{span} \{z_1, \dots, z_n\}$, 
then $x \perp z$.
This means that $z \in {\mathfrak L}$, i.e. $z = \sum_{k=1}^n \al_k z_k$ 
with suitable coefficients $\alpha_k$, $k=1, 2, \dots, n$.\end{proof}

\begin{lemma}\label{l54}
Let $\mathbf G$, $\mathbf A$ and $V$ satisfy \eqref{g:eq}, 
\eqref{lip:eq} and \eqref{01} with $p >3$.  
Let a function $w \in H^1 (\O)$ be such that
$\left. w \right|_{\Lambda_+} = 0$ and
$$
h_0 [w, Jv] = 0,  \quad \forall v \in W^1_+.
$$
Then $w=0$.
\end{lemma}

\begin{proof}
We extend the function $w$ by zero into the parallelepiped 
$\Xi = \L_- \times (-\pi,4)$:
$$
\tilde w (\mathbf x) = 
\begin{cases}
w(\mathbf x),\ \text{when}\ x_3 \le \pi, \\
0,\ \text{when}\ x_3 > \pi .
\end{cases}
$$
Clearly,
$\tilde w \in H^1 (\Xi)$, and
$$
\int_\Xi \left( \< \mathbf G(-i\n \tilde w 
- \mathbf A \tilde w), -i\n v - \mathbf A v \> 
+ V \tilde w \overline v \right) d\mathbf x 
= 0, \quad \forall v \in \mathring H^1 (\Xi) .
$$
Therefore $\tilde w$ is a weak solution of  the equation $H\tilde w = 0$ 
in $\Xi$. 
Now, the unique continuation principle for elliptic equations, see \cite{KT}, Theorem 1, 
implies that $\tilde w \equiv 0$ in $\Xi$.
\end{proof}

\begin{rem}
We need the conditions 
$\mathbf G\in \textup{Lip}$ and $A \in L_{p, \textup{\tiny loc}}$, $p>3$, 
instead of the "sharp" condition $A \in L_{3, \textup{\tiny loc}}$
for the unique continuation principle only.
\end{rem}

\begin{proof}[Proof of Theorem \ref{t52}]
By definition \eqref{kert1:eq}, for $u \in \ker t_1$ we have 
\begin{equation}\label{kernel:eq}
h_0 [u, J \phi] = 0 \quad \forall \phi \in Z.
\end{equation}
By Lemma \ref{Z:lem} the subspace $N$ is finite-dimensional. 
Let $\{u_k\}, k= 1, 2, \dots, n,$ be a basis in $N$.  
Consider on the Hilbert space $W^1_+$ linear functionals
$$
\ell(\psi) = \overline{h_0 [u, J \psi]}, \quad 
\ell_k(\psi) = \overline{h_0 [u_k, J  \psi]},\ \psi\in W^1_+.
$$
Since $JN = N$, by definition of $M$ we have  
$\cap_k \ker \ell_k = M$. On the other hand, if $\psi \in M$ then by 
Lemma \ref{l31} $\psi = \phi + w$ with $\phi\in Z, w\in W^{1, 0}$, 
so  
$$
h_0 [u, J \psi] = h_0 [u, J \phi] + h_0[u, Jw] = 0,
$$
where we have used \eqref{kernel:eq} and 
the fact that $u\in Z$. Thus $M \subset \ker \ell$.
By virtue of Lemma \ref{l53} there exists a function 
$u_0 \in N$ such that
$$
\ell(\psi) = \overline{h_0 [u_0, J \psi]}, \quad \forall\psi \in W^1_+.
$$
Therefore,
$$
h_0 [u-u_0, J \psi] = 0, \quad \forall \psi \in W^1_+.
$$
Putting $v=u-u_0$ we have
$h_0 [ v, J\psi] = 0$ for all $\psi \in W^1_+$.
By Lemma \ref{l54} $v = 0$, so that $u = u_0\in W^{1, 0}\cap Z$. By Lemma 
\ref{l31} $u = 0$ as claimed. 
\end{proof}

%%%%%%%%%%%%%%%%%%%%%%%%%%%%%%%%%%%%%%%%%%%%%%%%%%%%%%%%%%%%%%%%%%%%%%%%%%
\section{Proof of the main result}
\label{s6}

Recall that the operator $H(\mathbf k)$ depends on the quasi-momentum $\mathbf k$  
quadratically, i.e. it is 
a quadratic operator pencil. 
The decisive observation due to L. Friedlander \cite{Fr} is that 
the reflection symmetry allows one to reduce the 
analysis of $H(\mathbf k)$ to a linear operator pencil.

\subsection{An abstract lemma}
We will need the following abstract result. 
Let $\mathfrak H$ be a Hilbert space, 
and let $\mathfrak t$ 
be a bounded sesquilinear form defined on $\mathfrak H$. 
Similarly to \eqref{kert1:eq} we introduce the notation
\begin{equation*}
\ker \mathfrak t = \{\phi\in \mathfrak H: \mathfrak t[\phi, \psi] = 0,\ 
\forall \psi\in\mathfrak H\}.
\end{equation*}
The set $\ker\mathfrak t$ 
is a (closed) subspace. 
  
\begin{lemma}
\label{l61}
Let $\mathfrak H$ be a Hilbert space, and let $\mathfrak t_0$, $\mathfrak t_1$ 
be two bounded Hermitian sesquilinear forms on $\mathfrak H$. 
Let $\mathfrak L\subset \mathfrak H$ 
be a linear set such that $\mathfrak t_0[\phi]\le 0$ 
for any $\phi\in\mathfrak L$. 
Suppose that 
\begin{equation}\label{neg:eq}
m = \sup_{\mathfrak L} \dim\mathfrak L<\infty.
\end{equation}
Assume that $\ker \mathfrak t_1 = \{0\}$.
Then
\begin{equation*}
\# \left\{ z \in \C \setminus \R :
\ker (\mathfrak t_0 + z \mathfrak t_1) \neq \{0\} \right\} \le 2 m.
\end{equation*}
\end{lemma} 

Clearly this Lemma can be generalised to unbounded forms with appropriate 
restrictions on $\mathfrak t_0, \mathfrak t_1$ but it is unnecessary for our purposes. 

\begin{proof}
Let 
\begin{equation*}
F = \{z_1, z_2, \dots, z_n\},\ \im z_j >0, j = 1,2 , \dots, n,
\end{equation*}
be a finite set of distinct points in the complex plane such that 
\begin{equation*}
\mathfrak G_j = \ker (\mathfrak t_0+z_j \mathfrak t_1)\not 
= \{ 0\}, j = 1, 2, \dots, n. 
\end{equation*}
Let us show that 
the subspaces $\mathfrak G_j$ are linearly independent. We proceed by induction. 
If $n = 1$ then there is nothing to proof. 

Let $1\le p \le n-1$. 
Suppose that 
any $p$-tuple of non-zero vectors 
$\phi_k\in \mathfrak G_k, k = 1, 2, \dots, p$  are linearly-independent. Suppose 
also that $\phi_{p+1}\in\mathfrak G_{p+1}$ is a vector such that 
\begin{equation}\label{comb:eq}
\phi_{p+1} = \sum_{k=1}^p \alpha_k \phi_k,
\end{equation}
and at least one coefficient $\alpha_k$ is non-zero. 
By definition of $\mathfrak H_k$,
\begin{equation*}
\mathfrak t_0[\phi_k, w] + z_k \mathfrak t_1[\phi_k, w] = 0, 
\quad \forall w\in \mathfrak H, 
\end{equation*}
for all $k = 1, 2, \dots, p+1$.  
Therefore
\begin{equation*}
\sum_{k=1}^p\al_k \mathfrak t_0[\phi_k, w] 
+ \sum_{k=1}^p \al_k z_k \mathfrak t_1[\phi_k, w] = 0, 
\end{equation*}
and 
\begin{equation*}
\sum_{k=1}^p\al_k \mathfrak t_0[\phi_k, w] 
+ \sum_{k=1}^p \al_k z_{p+1} \mathfrak t_1[\phi_k, w] = 0,
\end{equation*}
for all $w\in \mathfrak H$, where we have used  \eqref{comb:eq}. Subtracting one equation from the other we get 
\begin{equation*}
\mathfrak t_1 \left[ \sum_{k=1}^{p} \al_k (z_k-z_{p+1}) \phi_k, w \right] = 0, 
\quad \forall \ w \in \mathfrak H.
\end{equation*}
Recalling again that $\ker \mathfrak t_1 = \{0\}$, we conclude that 
$$
\sum_{k=1}^{p} \al_k (z_k-z_{p+1}) \phi_k = 0,
$$
which means that the set $\{\phi_1, \phi_2, \dots, \phi_p\}$ is 
linearly dependent. This gives a contradiction, and hence the 
$(p+1)$-tuple containing also $\phi_{p+1}$ are linearly independent as well. 
By induction all kernels $\mathfrak G_j, j=1, 2, \dots, n$ are linearly-independent, 
and as a consequence, $\# F\le \dim \mathfrak G$ where 
\begin{equation*}
\mathfrak G = \bigoplus_{j=1}^n \mathfrak G_j.
\end{equation*}
Now, for any $\phi_j\in\mathfrak G_j, \phi_k\in \mathfrak G_k$, we have 
\begin{equation*}
\begin{cases}
\mathfrak t_0[\phi_j, \phi_k] + z_j\mathfrak t_1[\phi_j, \phi_k] = 0,\\[0.2cm]
\mathfrak t_0[\phi_j, \phi_k] + \overline z_k\mathfrak t_1[\phi_j, \phi_k] = 0, 
\end{cases}
\end{equation*}
where we have used that $\mathfrak t_0, \mathfrak t_1$ are Hermitian. 
Since $\im z_j, \im z_k >0$, we conclude that 
$\mathfrak t_0[\phi_j, \phi_k] = \mathfrak t_1[\phi_j, \phi_k] = 0$. As a consequence, 
\begin{equation*}
\mathfrak t_0[\phi, \psi] 
= \mathfrak t_1[\phi, \psi] = 0, \quad \forall \phi, \psi\in \mathfrak G. 
\end{equation*}
In particular, $\mathfrak t_0[\phi] = 0$, so that 
$\dim \mathfrak G\le m$, and hence, $\# F\le m$, i.e. 
\begin{equation*}
\# \left\{ z \in \C,\  \im z >0:
\ker (\mathfrak t_0 + z\mathfrak t_1) \neq \{0\} \right\} \le m.
\end{equation*} 
In the same way one proves that the number of 
such points in the lower half-plane is also bounded by $m$. 
This completes the proof. 
\end{proof}

Note in passing that if any of the forms $\mathfrak t_0$ or $\mathfrak t_1$ is positive-definite 
then the set 
\begin{equation}\label{edro:eq}
\{z\in \mathbb C\setminus\mathbb R: \ker (\mathfrak t_0 + z\mathfrak t_1) \not = \{0\}\}
\end{equation}
is trivially empty. Indeed, assume for example that $\mathfrak t_1$ is positive-definite. 
Let $T_0, T_1$ be the operators associated with the forms $\mathfrak t_0, \mathfrak t_1$ 
respectively. Thus $\ker (\mathfrak t_0 + z\mathfrak t_1) \not = \{0\}$ iff the number 
$z$ belongs to the spectrum of the self-adjoint operator $-T_1^{-\frac{1}{2}} T_2 T_1^{-\frac{1}{2}}$.
Thus $z\in \R$, which implies that the set \eqref{edro:eq} is empty, as claimed.

\subsection{Proof of Theorem \ref{t2}}
We begin the study of the problem 
\eqref{per3:eq}, \eqref{24} with the analysis 
of the following system for a function 
$u\in W^1$: 
\begin{equation}
\label{42}
\begin{cases}
h_0 [u,v] = 0, \quad \forall v \in W^{1,0}, \\
\left. u \right|_{\Lambda_+} = \ze \left. u \right|_{\Lambda_-}.
\end{cases}
\end{equation}

\begin{lemma}\label{l41}
 Suppose that the conditions \eqref{g:eq}, 
 \eqref{01} with $p=3$,  and \eqref{even:eq} are 
satisfied. 
Let $\ze \neq \pm 1$. Then any solution of \eqref{42} has the form 
\begin{equation}
\label{43}
u = \phi + \ze J \phi + \omega, 
\quad \text{where} \quad \phi\in Z, \ \omega \in N.
\end{equation}
\end{lemma}

\begin{proof}
Let $u$ be a solution to \eqref{42}. 
Then the function 
$\psi = (1-\ze^2)^{-1} (u-\ze Ju)$ belongs to $W^1_+$ and solves 
the equation $h_0[\psi, v] = 0,\ \forall v\in W^{1, 0}$,  
and hence $\psi \in M$. 
By Lemma \ref{l31}, $\psi = \phi+ w$ where $\phi\in Z$ and $w \in W^{1, 0}$. 
Consequently $w, Jw\in N$. 
By inspection, 
\begin{equation}\label{inver:eq}
u = \psi + \ze J \psi,
\end{equation}
so that the representation \eqref{43} holds with $\omega = w+\ze J w\in N$. 
\end{proof}

\begin{proof}[Proof of Theorem \ref{t2}] 
Let $\ze\in X$, and let $u\in W^1$ 
be a non-trivial solution of the system  
\eqref{per3:eq}, \eqref{24}. 
By virtue of Lemma \ref{l41},
$u = \phi + \ze J \phi + \omega$, 
with some $\phi\in Z$ and $\omega \in N$.

First, consider the case $\phi = 0$.
Then $u =\omega\in N$. Let us use \eqref{24} 
with the function $w = \overline\ze  f + J f$ 
where $f \in W^1_+$ is an arbitrary function. Thus 
$$
h_0 [\ze J u + u, J f] = h_0 [u, \overline\ze f + J f] = 0 .
$$
Lemma \ref{l54} yields
$J u = - \ze^{-1} u$. 
On the other hand, the spectrum of the operator $J$ 
consists of two numbers $1$ and $-1$ only, but $|\ze|\not = 1$, so $u = 0$, 
which gives a contradiction.

Now, assume that $\phi \neq 0$.
For a function $\psi \in Z$ substitute
$w = J \psi + \overline\ze \psi$
into \eqref{24}, and obtain
\begin{eqnarray*}
0 = h_0 [u, w] = 
h_0 \left[\phi + \ze J \phi + \omega, 
J \psi + \overline\ze  \psi \right] \\
= \left(1+\ze^2\right) t_1 [\phi, \psi] + 2 \ze t_0 [\phi, \psi],
\end{eqnarray*}
where we have used the fact that $J W^{1, 0} = W^{1, 0}$. 
Therefore,
$$
t_0 [\phi, \psi] + z t_1 [\phi, \psi] = 0 ,
\quad z = \frac{1+\ze^2}{2\ze}.
$$
In view of the conditions $\im \ze\not = 0$, $|\ze|\not=1$ we have  
$\im z\not = 0$. 
The forms $t_0, t_1$ satisfy all the conditions of Lemma \ref{l61}. 
Indeed, both forms are bounded on $Z$, 
$\ker t_1 = \{0\}$ by Theorem \ref{t52}, and 
the condition \eqref{neg:eq} is satisfied by virtue of \eqref{vari:eq}.
Therefore Lemma \ref{l61} yields that $\# X \le 2m <\infty$. 
This completes the proof. 
\end{proof}

As explained earlier, Theorem \ref{t2} implies Theorem \ref{t1} 
stating the absolute continuity 
of the operator $H$.

%%%%%%%%%%%%%%%%%%%%%%%%%%%%%%%%%%%%%%%%%%%%%%%%%%%%%%%%%%%%%%%%%%%

\end{document}